\documentclass[12pt]{amsart}
\usepackage[a4paper]{geometry}
\usepackage{inputenc}
\usepackage{amsmath}
\usepackage{amsfonts}
\usepackage{amssymb}
\usepackage{amsthm}
\usepackage{verbatim}

\theoremstyle{definition}
\newtheorem{defi}{Definition}

\newtheorem*{rem}{Remark}
\newtheorem{cor}{Corollary}

\theoremstyle{plain}
\newtheorem{thm}{Theorem}
\newtheorem{lem}[thm]{Lemma}
\newtheorem{prop}[thm]{Proposition}

\DeclareMathOperator{\ric}{Ric}
\DeclareMathOperator{\tr}{tr}

\begin{document}

\title[$\mathcal{A}$-manifold on principal torus bundle]{Construction of an $\mathcal{A}$-manifold on a principal torus bundle}
\author[G. Zborowski]{Grzegorz Zborowski}
\address{Cracow University of Technology\\ Warszawska 24\\ 31-155 Krak\'ow, Poland}
\address{University of Maria Curie-Skłodowska\\ Pl. Marii Curie-Skłodowskiej 5, 20-035 Lublin, Poland}
\email{gzborowski@pk.edu.pl}

\begin{abstract}
We construct a new example of an $\mathcal{A}$-manifold, i.e. a Riemannian manifold with cyclic-parallel Ricci tensor. This condition can be viewed as a generalization of the Einstein condition. The underlying manifold for our construction is a principal torus bundle over a product of K\"ahler-Einstein manifolds with fibre a torus of arbitrary dimension. 
\end{abstract}

\thanks{The author would like to thank prof. W. Jelonek for the ideas and time spent to make this work right.}

\maketitle

\section{Introduction}

In this paper we would like to present a construction of a special kind of Riemannian manifold called $\mathcal{A}$-manifold.
\begin{defi}
A Riemannian manifold $(M,g)$ is called an \emph{$\mathcal{A}$-manifold} iff the Ricci tensor $\rho$ of $(M,g)$ satisfies
\begin{displaymath}
\nabla_X\rho(X,X) = 0,
\end{displaymath}
for any $X\in TM$. Moreover if the Ricci tensor $\rho$ is not parallel then $(M,g)$ is called a \emph{proper $\mathcal{A}$-manifold}.
\end{defi}
This condition was first considered by A. Gray in \cite{gray} and then it appeared in \cite{besse} where the problem of finding non-homogenous examples with non-parallel Ricci tensor was given. 

Recall that Einstein manifold satisfies the condition that the Ricci tensor is a constant multiple of the Riemannian metric. It is now easy to see that every Einstein manifold is an $\mathcal{A}$-manifold. The first example of an $\mathcal{A}$-manifold which is not Einstein was given in \cite{gray}. The first non-homogeneous example was given in \cite{j1}. The main result (\cite{j1}, Theorem $3.3$) can be stated as
\begin{thm}
On a $S^1$-principal bundle $S$ over a K\"ahler-Einstein manifold, classified by the K\"ahler class of the base, with non-vanishing scalar curvature we can construct many Riemannian metrics $g_c$, parametrized by a positive number $c$, such that $(S,g_c)$ is an $\mathcal{A}$-manifold. Moreover for some choices of $c$ $(P,g_c)$ is a proper $\mathcal{A}$-manifold.
\end{thm}
This article gives a generalization of this theorem and the main result (Theorem \ref{main}) of the present work can be informally stated as
\begin{thm}
Given a principal $r$-torus bundle $P$ over a product of K\"ahler-Einstein manifolds $B_i$, $i=1,\ldots,m$, classified by K\"ahler classes of $B_i$, with positive first Chern classes we can construct a Riemannian metric $g$ such that $(P,g)$ is an $\mathcal{A}$-manifold.
\end{thm} 
Our construction is strongly motivated by \cite{wang} where authors prove the existence of Einstein metrics with positive scalar curvature on principal torus bundles. Also in \cite{pet} authors construct a principal torus bundle with total space being an $\mathcal{A}$-manifold but their construction is different and they obtain only examples with two eigenvalues of the Ricci tensor where in the present paper we give examples with any number of eigenvalues. 

\section{Preliminaries}

We introduce some notational conventions used along the work. Let $M,P,B$ be smooth differential manifolds and $k\in\mathbb{Z}$.
\begin{itemize}
\item $A^k(M)$ - set of differential $k$-forms on $M$.
\item If $p\,:\,P\rightarrow B$ is a fibre bundle we will refer to it by writing only $P$ for brevity.
\item For a $(1,1)$-tensor field $S$ denote by $\nabla S(X,Y)$ the covariant derivative $\nabla_XS(Y)$.
\item $[X,Y]$ denotes the Lie bracket of two vector fields.
\end{itemize}

In the rest of this section we prove a theorem which will help us to construct some $\mathcal{A}$-manifolds in the last section of this work.

Let $(M,g)$ be a Riemannian manifold and $\mathfrak{t}^* = \{\xi^1,\ldots,\xi^k\}$ be a set of linearly independent Killing vector fields on $M$ such that $g(\xi^i,\xi^j) = const$, $[\xi^i,\xi^j] = 0$ and $\mathfrak{t}^*$ gives a decomposition of $TM$ into distribution spanned by $\mathfrak{t}^*$ and its orthogonal complement. Define $T_iX = \nabla_X\xi^i$, $i=1,\ldots,k$.
\begin{lem}\label{kill}
Tensor $T_i$ satisfies, for $i,j=1,\ldots,k$,
\begin{align*}
T_{i}\xi^j &= 0,\\
L_{\xi^i}T_j &= 0.
\end{align*}
\end{lem}
\begin{proof}
For any $\xi^i,\xi^j\in\mathfrak{t}^*$ we have $g(\xi^i,\xi^j) = const$, hence for any $X\in TM$ we have $X(g(\xi^i,\xi^j)) = 0$ and
\begin{displaymath}
0 = g(\nabla_X\xi^i,\xi^j) + g(\xi^i,\nabla_X\xi^j) = -g(X,\nabla_{\xi^j}\xi^i) - g(\nabla_{\xi^i}\xi^j,X).
\end{displaymath}
Since $[\xi^i,\xi^j] = 0$ we have $\nabla_{\xi^i}\xi^j = \nabla_{\xi^j}\xi^i$. This ends the proof of the first equality.

As for the second, recall that for any Killing vector field $X$ on $M$ and any vector fields $Y,Z$ we have
\begin{displaymath}
L_X(\nabla_YZ) = \nabla_{L_XY}Z + \nabla_Y(L_XZ).
\end{displaymath}
Hence, for any vector field $X$ we have
\begin{displaymath}
(L_{\xi^i}T_j)X = L_{\xi^i}(T_jX) - T_j(L_{\xi^i}X) = \nabla_{(L_{\xi^i}X)}\xi^j + \nabla_X([\xi^i,\xi^j]) - T_j(L_{\xi^i}X) = 0.
\end{displaymath}
\end{proof}

We can state a lemma and a corollary that will be of use in what will follow.
\begin{lem}
With the notation from above we have
\begin{displaymath}
R(X,\xi^i)Y = \nabla T_i(X,Y),\quad \nabla T_i(X,\xi^j) = -T_i(T_jX),
\end{displaymath}
where $R$ is the curvature tensor of $(M,g)$ and $X,Y\in TM$.
\end{lem}
\begin{proof}
The proof consists in rewriting the analogous proof from \cite{j1} for $k$ Killing fields. Recall, that for any Killing vector field $\xi$ we have
\begin{displaymath}
R(X,\xi) = \nabla_X(\nabla\xi).
\end{displaymath}
Hence we have nothing to prove as
\begin{displaymath}
R(X,\xi^i)Y = \nabla_X T_i(Y).
\end{displaymath}
Since $T_i\xi^j = 0$ for $i,j=1,\ldots,k$ we have $\nabla_X T_i(\xi^j) + T_i(\nabla_X\xi^j) = 0$ and there is also nothing to prove.
\end{proof}

\begin{cor}\label{ric}
We have
\begin{displaymath}
\rho(\xi^i,X) = -g(X,\tr_g\nabla T_i),
\end{displaymath}
for $\xi^i$ from above and any $X\in TM$.
\end{cor}

We will now prove a theorem which states certain conditions for a Riemannian manifold to be an $\mathcal{A}$-manifold. In the proof we use an equivalent condition characterizing $\mathcal{A}$-manifolds (\cite{gray},\cite{j1}):
\begin{equation}\label{a-war}
g(\nabla\ric(X,Y),Z) + g(\nabla\ric(Y,Z),X) + g(\nabla\ric(Z,X),Y) = 0,
\end{equation}
for all vector fields $X,Y,Z$, where the Ricci endomorphism $\ric$ is defined by $g(\ric X,Y) = \rho(X,Y)$.

\begin{thm}\label{a-tw}
Let $(M,g)$ be as in Lemma \ref{kill} and $\ric$ the Ricci endomorphism. Assume that $\mu\in\mathbb{R}$ is an eigenvalue of $\ric$ and define $\mathcal{H} = \mathrm{ker}\,(\ric - \mu \mathrm{Id})$. Furthermore, assume that the orthogonal complement $\mathcal{H}^{\perp}$ of $\mathcal{H}$ with respect to $g$ is generated by $\mathfrak{t}^*$ and for every $i,j=1,\ldots,k$ we have $\rho(\xi^i,\xi^j) = const$. Then $(M,g)$ is an $\mathcal{A}$-manifold.
\end{thm}
\begin{proof}
Observe first that if $X\in TM$, $Y\in\mathcal{H}$ and $\xi^i\in\mathfrak{t}^*$ then we have
\begin{gather*}
\ric Y = \mu Y,\quad \nabla\ric(X,Y) = -(\ric -\mu\mathrm{Id})(\nabla_XY),\\
\ric \xi^i = \sum_{j=1}^kc^i_j\xi^j,
\end{gather*}
for some constants $c^i_j$, $i,j=1,\ldots,k$.

We will now check the condition \eqref{a-war} for different choices of vector fields. First assume, that $X,Y,Z\in\mathcal{H}$. Then
\begin{displaymath}
g(\nabla\ric(X,Y),Z) = g(-(\ric -\mu\mathrm{Id})(\nabla_XY), Z) = 0,
\end{displaymath}
where we use the fact that $Z\in\mathcal{H}$.

Now assume, that $X,Y\in\mathcal{H}$ and $Z\in\mathcal{H}^{\perp}$. We can additionally assume that $Z = \xi^i\in\mathfrak{t}^*$, hence it is Killing. We have $g(\nabla\ric(Y,Z),X) = g(\nabla\ric(Y,X),Z)$ and
\begin{displaymath}
g(\nabla\ric(Z,X),Y) =  g(-(\ric -\mu\mathrm{Id})(\nabla_ZX), Y) = 0,
\end{displaymath}
so remaining components in the cyclic sum \eqref{a-war} are
\begin{gather*}
g(\nabla\ric(X,Y),Z) + g(\nabla\ric(Y,X),Z) \\
= g(-(\ric -\mu\mathrm{Id})(\nabla_XY), Z) + g(-(\ric -\mu\mathrm{Id})(\nabla_YX), Z)\\
= \mu\left(g(\nabla_XY,Z) + g(\nabla_YX,Z)\right) - \left(g(\nabla_XY,\sum_{j=1}^kc^i_j\xi^j) + g(\nabla_YX,\sum_{j=1}^kc^i_j\xi^j)\right) \\
= - \mu\left(g(Y,\nabla_XZ) + g(X,\nabla_YZ)\right) + \sum_{j=1}^kc^i_j \left(g(Y,\nabla_X\xi^j) + g(X,\nabla_Y\xi^j)\right) = 0,
\end{gather*}
because $X,Y$ are perpendicular to every $\xi^j$ and $\xi^j$ are Killing, $j=1,\ldots,k$.

Now let $X =\xi^i,Y=\xi^j$ and $Z\in TM$. We have
\begin{align*}
g(\nabla\ric(X,Y),Z) &= g(\nabla_X(\ric Y),Z) - g(\ric(\nabla_XY),Z)\\
 &= \sum_{l=1}^kc^j_lg(T_lX,Z) - \sum_{l=1}^kc^j_lg(T_lX,\ric Z) = 0
\end{align*}
by Lemma \ref{kill}. The same argument is valid for $g(\nabla\ric(Y,Z),X)$ because of the symmetry of $\nabla\ric$. For the last summand of the cyclic sum we have
\begin{align*}
g(\nabla\ric(Z,X),Y) &= g(\nabla_Z(\ric X),Y) - g(\ric(\nabla_ZX),Y)\\
&= \sum_{l=1}^kc^i_l g(\nabla_Z\xi^l,Y) - \sum_{l=1}^kc^j_l g(\nabla_ZX,\xi^l)\\
&= -\sum_{l=1}^kc^i_l g(Z,T_lY) + \sum_{l=1}^kc^j_l g(Z,T_i\xi^l) = 0
\end{align*}
again by Lemma \ref{kill}. This proves the vanishing of the cyclic sum for $Z\in\mathcal{H}$ and $Z\in\mathfrak{t}^*$ and finishes the proof.
\end{proof}

\section{Principal torus bundles}

Let $(B,h)$ be a K\"ahler-Einstein manifold of real dimension $n$. We will first recall some facts about principal $S^1$-bundles. If $p:S\rightarrow B$ is a principal circle bundle, then denote by $\theta\in A^1(S)$ its connection form and by $\xi$ its fundamental vector field, i.e. $\theta(\xi) = 1$. Thanks to the fact that the Lie algebra of $S^1$ is the real line $\mathbb{R}$ this connection form is a real-valued differential form on $S$. Denote by $\Omega = d\theta$ the curvature form of $S$. Since $\Omega$ is projectable, there exists a closed $2$-form $\omega$ on $B$, such that $\Omega = p^*\omega$.

For some positive function $f$ on $B$ we can define a metric on $S$ by
\begin{displaymath}
g_f(X,Y) = f^2\theta(X)\theta(Y) + p^*h(X,Y),
\end{displaymath}
where $p^*$ denotes the pullback map induced by $p$.

With this metric the map $p:(S,g_f)\rightarrow (B,h)$ becomes a Riemannian submersion with totally geodesic fibres. Moreover, assume that $f$ is a constant and $\omega=-\frac{\tau}{n}\eta$, where $\tau$ is the scalar curvature of $(B,h)$ and $\eta$ is the pullback of the K\"ahler form. Then $(S,g_f)$ becomes an $\mathcal{A}$-manifold (see \cite{j1} for details).

Let $p:P\rightarrow B$ be the Whitney sum of principal $S^1$-bundles $P_1$,\ldots, $P_r$. Then $P$ is a $T^r$-principal bundle, where $T^r$ is a $r$-dimensional torus. The connection form of $P$ is a differential form with values in the Lie algebra $\mathcal{L}(T^r)$ of $T^r$, namely $\mathbb{R}^r$. Components $\theta^1,\ldots,\theta^r$ of this connection form are pullbacks of connection forms on $P_1,\ldots, P_r$. With this comes $r$ fundamental fields $\xi^1,\ldots, \xi^r$ which are pullbacks of the fundamental fields on $P_1,\ldots,P_r$.

Similarly we can write the curvature form $\Omega$ of $P$ as $\Omega_1e_1 +\ldots + \Omega_re_r\in\mathcal{L}(T^r)$ where $\Omega_i$ is the curvature form of the $S^1$ bundle $P_i$ and $e_1,\ldots,e_r$ are basis vectors for $\mathbb{R}^r$. Hence, there exists $r$ closed $2$-forms $\omega_i$ on $B$ such that $\Omega_i = p_i^*\omega_i$. Pulling $(\omega_1,\ldots,\omega_r)$ back along the diagonal map, we get a single $2$-form $\omega$ on $B$ such that $\Omega = p^*\omega$.

We now introduce a metric on $P$ which will make $p:P\rightarrow B$ a Riemannian submersion
\begin{displaymath}
g(X,Y) = \sum_{i,j=1}^r{b_{ij}\theta^i(X)\theta^j(Y)} + p^*h(X,Y),
\end{displaymath}
where $[b_{ij}]_{i,j=1}^r$ is a positive definite symmetric matrix. This matrix induces a left-invariant metric on $T^r$.

\begin{prop}
In the above situation $\xi^1,\ldots,\xi^r$ are Killing vector fields with respect to the metric $g$.
\end{prop}
\begin{proof}
Let $\xi^k$ be one of vector fields defined above as pullbacks of fundamental fields on $P_k$. First observe, that the Lie derivative $L_{\xi^k}g$ depends only on the Lie derivatives of $\theta^l$ with respect to $\xi^k$
\begin{displaymath}
L_{\xi^k}g = \sum_{i,j=1}^r{b_{ij}\left(L_{\xi^k}(\theta^j)\otimes\theta^i) + \theta^j\otimes L_{\xi^k}(\theta^i)\right)}.
\end{displaymath}
As $L_{\xi^k}\theta^k = 0$ only non-vanishing parts of the above derivative are
\begin{displaymath}
L_{\xi^k}\theta^l = d(i_{\xi^k}\theta^l) + i_{\xi^k} d\theta^l,\quad l\neq k.
\end{displaymath}
Since $\theta^l(\xi^k) = 0$ we only have to check values of the $1$-form $d\theta^l(\xi^k,X)$ for any $X$ in $TP$. We have
\begin{displaymath}
d\theta^l(\xi^k,X) = \xi^k(\theta^l(X)) - X(\theta^l(\xi^k)) - \theta^l([\xi^k,X]),
\end{displaymath}
where the second term is immediately zero. 

Assume first, that $X$ belongs to the horizontal distribution of the connection on $P$. Then $\theta^l(X) = 0$ and $[\xi^k,X]$ is a horizontal vector field, as $\xi^k$ is a fundamental vector field. Hence the right hand side vanishes. 

Next, assume that $X$ is vertical. As $d\theta^l$ is $C^{\infty}(P)$-multilinear we can assume that $X$ is just a linear combination of $\xi^1,\ldots,\xi^r$. Hence $\theta^l(X)$ is constant and the differential vanishes. The Poisson bracket $[\xi^k,X]$ also vanishes as $T^r$ is abelian. This finishes the proof.
\end{proof}

As a corollary of Lemma \ref{kill} we get a property of tensors $T_i$, $i=1,\ldots,r$.
\begin{cor}
The tensor $T_i$ is horizontal, i.e. there exist a tensor $\tilde{T}_i$ on $B$ such that $\tilde{T}_i\circ p_* = p_*\circ T_i$.
\end{cor}

\section{O'Neill fundamental tensors}

We would like to compute the O'Neill tensor (see \cite{oneill}) $A$ of the Riemannian submersion $p\;:\; P\rightarrow B$. First observe that fibres are totally geodesic, hence the O'Neill tensor $T$ is zero.

\begin{prop}
The vertical distribution of the Riemannian submersion $p:P\rightarrow B$ is totally geodesic.
\end{prop}
\begin{proof}
We have to check that $g(\nabla_UV,X) = 0$ for all vertical vector fields $U,V$ and any horizontal vector field $X$. Since we are interested only in horizontal part of $\nabla_UV$ we can assume that $U = \xi^j$, $V= \xi^i$ where $i,j\in\{1,\ldots,r\}$, since $\nabla(U,V) = \nabla_UV$ is tensorial in the first variable and
\begin{displaymath}
\nabla_U(fV) = df(U)V + \nabla_UV
\end{displaymath}
for any smooth function $f$.

From $T_j\xi^i=0$ we get for any horizontal field $X$
\begin{displaymath}
2g(\nabla_{\xi^i}\xi^j,X) =  0,
\end{displaymath}
since the Poisson bracket of any fundamental and vertical vector field is horizontal.
\end{proof}

Now we proceed to the O'Neill tensor $A$.

\begin{lem}\label{atens}
For any vector fields $E,F$ on $P$ we have
\begin{displaymath}
A_{E}F = \sum^{r}_{i,j=1}{b^{ij}\left(g(E,T_iF)\xi^{j} + g(\xi^i,F)T_{j}E\right)},
\end{displaymath}
where $b^{ij}$ are coefficients of the inverse matrix of $[b_{ij}]_{i,j=1}^r$.
\end{lem}
\begin{proof}
From \cite{oneill} we have
\begin{displaymath}
A_{E}F = \mathcal{V}\nabla_{\mathcal{H}E}\mathcal{H}F + \mathcal{H}\nabla_{\mathcal{H}E}\mathcal{V}F,
\end{displaymath}
where $\mathcal{V}$ and $\mathcal{H}$ denote the vertical and the horizontal part of a vector field respectively.

For any horizontal vector fields $X,Y$ we have
\begin{displaymath}
g(\nabla_XY,\xi^i) = Xg(Y,\xi^i) - g(Y,\nabla_X\xi^i) = g(X,T_iY),
\end{displaymath}
for $i=1,\ldots,r$.

The vertical part of any vector field $F$ is given by
\begin{displaymath}
\mathcal{V}F = \sum^{r}_{i,j=1}{b^{ij}g(\xi^i,F)\xi^{j}}.
\end{displaymath}
Hence
\begin{align*}
g\left(\mathcal{H}E,T_{k}\mathcal{H}F\right) &= g\left(E - \sum^{r}_{i,j=1}{b^{ij}g(\xi^i,E)\xi^{j}},T_{k}\left(F - \sum^{r}_{i,j=1}{b^{ij}g(\xi^i,F)\xi^{j}}\right)\right) \\
&= g(E,T_{k}F).
\end{align*}
Combining all of above formulae we get
\begin{displaymath}
\mathcal{V}\nabla_{\mathcal{H}E}\mathcal{H}F = \sum_{i,j=1}^r{b^{ij}g(E,T_iF)\xi^j}.
\end{displaymath}

To compute the second summand of the O'Neill tensor $A$ we observe, that for any function $f$ and any $E\in TP$ we have $\mathcal{H}\nabla_{E}f\xi^{i} = \mathcal{H}f\nabla_{E}\xi^{i}$ for $i=1,\ldots,r$. Then
\begin{align*}
\mathcal{H}\nabla_{\mathcal{H}E}\mathcal{V}F &= \mathcal{H}\nabla_{E - \sum^{r}_{i,j=1}{b^{ij}g(\xi^i,F)\xi^{j}}}\sum^{r}_{i,j=1}{b^{ij}g(\xi^i,F)\xi^{j}}\\
 &= \sum^{r}_{i,j=1}{b^{ij}g(\xi^i,F)T_{j}E}.
\end{align*}
\end{proof}

\section{Conditions for $P$ to be an $\mathcal{A}$-manifold}

We will check under what conditions a principal $T^r$-bundle $P$ with some Riemannian metric $g$ satisfies assumptions of Theorem \ref{a-tw}. An example of a Riemannian manifold satisfying those conditions will be given in the next section. The distribution $\mathcal{H}$ from the theorem will be the horizontal distribution of the submersion.

The first condition we will reinterpret is $\rho(X,V) = 0$ for $X\in\mathcal{H}$ and $V$ in its orthogonal complement. From Corollary \ref{ric} we get
\begin{displaymath}
0 = \rho(\xi^i,X) = -g(X,\tr_g\nabla T_i),\quad i=1,\ldots,r.
\end{displaymath}

Recall, that for $i=1,\ldots,r$ we have $\Omega^i = p^*\omega^i$, where $\omega^i$ is a $2$-form on $B$ and $\Omega^i = d\theta^i$. We obtain an explicit formula for $\theta^i$. First observe that for any $X\in TP$
\begin{displaymath}
g(\xi^k,X) = \sum_{j=1}^r{b_{kj}\theta^j(X)},\quad k=1,\ldots,r.
\end{displaymath}
If we solve this for $\theta^j$ we get
\begin{displaymath}
\theta^j(X) = \sum_{i=1}^r{b^{ji}g(\xi^i,X)},\quad j=1,\ldots,r,
\end{displaymath}
where $b^{ji}$ are the coefficients of the inverse matrix of $[b_{ji}]_{i,j=1}^r$. Taking the differential we get
\begin{displaymath}
\Omega^i(X,Y) = d\theta^i(X,Y) = 2\sum_{j=1}^r{b^{ij}g(T_jX,Y)},
\end{displaymath}
where $X,Y\in TP$.
As $T_i$ are horizontal for $i=1,\ldots,r$ we get
\begin{equation}\label{o1}
\omega^i(p_*X,p_*Y) = 2\sum_{j=1}^r{b^{ij}h(\tilde{T}_j(p_*X),p_*Y)},
\end{equation}
where $\tilde{T}_i$ were defined in Lemma \ref{kill}. The co-differential of $\omega^i$ is given by 
\begin{displaymath}
(\delta\omega^i)(p_*X) = -\sum_{k=1}^m{(\nabla^h_{E_k}\omega^i)(E_k,p_*X)},
\end{displaymath}
where $\{E_k\}_{k=1}^m$ is some orthonormal basis of $B$. Hence
\begin{displaymath}
(\delta\omega^i)(p_*X) = -2\sum_{j=1}^r{b^{ij}h(\tr_h\nabla^h\tilde{T}_j,p_*X)},\quad i=1,\ldots,r.
\end{displaymath}
By a straightforward computation we get for every $i=1,\ldots,r$
\begin{align*}
h(p_*X,\tr_h\nabla^h\tilde{T}_i) &= \sum_{k=1}^m{\left(h(p_*X,\nabla^h_{E_k}(\tilde{T}_iE_k)) - h(p_*X, \tilde{T}_i(\nabla^h_{E_k}E_k))\right)}\\
&= \sum_{k=1}^m{\left(g(X,\nabla_{X_k}(T_iX_k)) - g(X, T_i(\nabla_{X_k}X_k)\right)}\\
&= g(X,\tr_g\nabla T_i).
\end{align*}
We easily see, that if the condition $\rho(V,X)=0$ is satisfied, then $\omega^i$ is harmonic for $i=1,\ldots,r$. Moreover, from equations for co-differentials of $\omega^i$ we see that if all $\omega^i$, $i=1,\ldots,r$, are harmonic then from non-singularity of the matrix $[b_{ij}]$ we get that $h(\tr_h\nabla^h\tilde{T}_1,p_*X) = \ldots = h(\tr_h\nabla^h\tilde{T}_r,p_*X) = 0$. This means that $g(X,\tr_g\nabla T_i) = 0$, $i=1,\ldots,r$. Hence the condition $\rho(V,X) = 0$ is satisfied iff 
\begin{equation}\label{orto}
\omega^i \text{ are harmonic for every } i\in\{1,\ldots,r\}.
\end{equation}
\begin{rem}
Observe that this is the Yang-Mills condition for the principal connection $\theta$ of the bundle $P$.
\end{rem}

The next condition is $\rho(\xi^i,\xi^j) = const$ for $\xi^i$ as above. We reformulate it as following. Again using Corollary \ref{ric} we have $\rho(\xi^i,\xi^j) = -g(\xi^j,\tr_g\nabla T_i),i,j\in\{1,\ldots,r\}$. This means that
\begin{displaymath}
-g(\xi^j,\tr_g\nabla T_i) = const.
\end{displaymath}
But for any orthonormal basis $\{X_k\}_{k=1}^m$ of the horizontal distribution $\mathcal{H}$ and $i\in\{1,\ldots,r\}$ we have $\tr_g\nabla T_i = \sum_{k=1}^m{(\nabla T_i)(X_k,X_k)}$, so for $j=1,\ldots,r$ we get
\begin{displaymath}
g(\xi^j,\tr_g\nabla T_i) = -\sum_{k=1}^m{g(\nabla T_i(X_k,\xi^j),X_k)}.
\end{displaymath} 
As $\nabla T_i(X_k,\xi^j) = -T_iT_jX_k$ we have
\begin{displaymath}
\sum_{k=1}^m{g(T_jX_k,T_iX_k)} = const,\quad i,j\in\{1,\ldots,r\}
\end{displaymath}
and from the fact that $T_i$ is horizontal for $i=1,\ldots,r$ we get
\begin{equation}\label{lambda}
\sum_{k=1}^m{h(\tilde{T}_j(X_k)_*,\tilde{T}_i(X_k)_*)} = const,\quad i,j\in\{1,\ldots,r\},
\end{equation}
where $X_*$ is the projection of the vector field $X$ from $M$ to $B$. It follows that we can find a basis $\{\zeta^1,\ldots,\zeta^r\}$ of the vertical distribution such that $\ric \zeta^i = \lambda_i\zeta^i$, where $\lambda_i\in\mathbb{R}$ and $i=1,\ldots,r$.

Now we consider the equivalent condition for
\begin{displaymath}
\rho(X,Y) = \mu g(X,Y)
\end{displaymath}
where $X$ and $Y$ are in $\mathcal{H}$.

First let us recall that the tensor $A$ for two horizontal vector fields $X$ and $Y$ is given by
\begin{displaymath}
A_XY = \sum_{i,j=1}^r{b^{ij}g(X,T_iY)\xi^j}.
\end{displaymath}
We use the formula for the Ricci tensor $\rho$ of the total space of Riemannian submersion from \cite{besse}:
\begin{displaymath}
\rho(X,Y) = \check{\rho}(X,Y) -2\sum_{k=1}^m{g(A_XX_k,A_YX_k)},
\end{displaymath}
where $\check{\rho}$ is the horizontal symmetric $2$-tensor given by
\begin{displaymath}
\check{\rho}(X,Y) = \rho_B(p_*X,p_*Y)
\end{displaymath}
and $\rho_B$ is the Ricci tensor of $B$. We can rewrite this as
\begin{align*}
\rho(X,Y) &= \check{\rho}(X,Y) - 2\sum_{k=1}^m{g(\sum_{i,j=1}^r{b^{ij}g(X,T_iX_k)\xi^j}, \sum_{s,t=1}^r{b^{st}g(Y,T_sX_k)\xi^t})}\\
&= \check{\rho}(X,Y) - 2\sum_{k=1}^m{g(\sum_{i,j=1}^r{b^{ij}g(T_iX,X_k)\xi^j}, \sum_{s,t=1}^r{b^{st}g(T_sY,X_k)\xi^t})}\\
&= \check{\rho}(X,Y) - 2\sum_{k=1}^m{\sum_{i,j=1}^r{b^{ij}\sum_{s,t=1}^r{b^{st}b_{jt}g(T_iX,X_k)g(T_sY,X_k)}}}\\
&= \check{\rho}(X,Y) - 2\sum_{k=1}^m{\sum_{i,j=1}^r{b^{ij}g(T_iX,g(T_jY,X_k)X_k)}}\\
&= \check{\rho}(X,Y) - 2\sum_{i,j=1}^r{b^{ij}g(T_iX,T_jY)}.
\end{align*}

As $\rho$ is symmetric it is determined by its values on the diagonal
\begin{equation}\label{mu}
\rho(X,X) = \check{\rho}(X,X) - 2\sum_{i,j=1}^r{b^{ij}g(T_iX,T_jX)}.
\end{equation}
Fortunately we can express this condition in terms of the metric on the base $B$ as $\check{\rho}$ and $T_i$ are horizontal for $i=1,\ldots,r$.

\section{Construction of an example over K\"ahler-Einstein base}

In this section we will construct an $\mathcal{A}$-manifold on a torus bundle over a product of K\"ahler-Einstein manifolds. We will follow \cite{j1} and \cite{wang}.

Let $(B,g_B)$ be a compact K\"ahler manifold with positive definite Ricci tensor. By a theorem of S. Kobayashi (\cite{koba1}) it is simply connected. Moreover, by another theorem from the same work if $B$ has positive first Chern class then $H_1(B;\mathbb{Z}) = 0$. Hence if $B$ is K\"ahler-Einstein with positive scalar curvature the above holds true. From the Universal Coefficient Theorem for Cohomology we get that $H^2(B;\mathbb{Z})$ has no torsion. Then we can write the first Chern class $c_1(B)$ as an integer multiple of some indivisible class $\alpha\in H^2(B;\mathbb{Z})$, say $c_1(B) = q\alpha$, $q\in\mathbb{Z}_+\setminus \{0\}$. We can normalize the K\"ahler metric $g_B$ so that the cohomology class of $\omega_B$ is the same as that of $2\pi\alpha$ which is equivalent to choosing the Einstein constant to be equal to $q$. 

By another theorem of Kobayashi (\cite{koba2}) every principal $S^1$-bundle $S$ over a Riemannian manifold $(B,g_B)$ is classified by a cohomology class in $H^2(B;\mathbb{Z})$. In fact this class is just the Euler class $e(S)$ of $S$. Hence every principal $T^r$-bundle $p:P\rightarrow B$ is classified by $r$ cohomology classes $\beta_1,\ldots,\beta_r$ in $H^2(B;\mathbb{Z})$. Those classes can be described as Euler classes of the quotient principal $S^1$-bundle $P/S_i$ over $B$, where $S_i$ is the sub-torus of $T^r$ with the $i$-th $S^1$ factor deleted in the canonical decomposition $T^r = S^1\times\ldots\times S^1$, where we take the $r$-fold product of $S^1$.

\begin{thm}\label{main}
Let $(B_j,g_j)$ be compact K\"ahler-Einstein manifolds with $c_1(B_j)>0$ for $j=1,\ldots,m$ and define a K\"ahler manifold $B = B_1\times\ldots\times B_m$ with metric $h = \sum_{j=1}^m{x_jg_j}$ where $x_j$ are some positive constants. Let $p:P\rightarrow B$ be a principal $T^r$ bundle characterised by $\beta_i = \sum_{j=1}^m{a_{ij}pr_j^*\alpha_j}$ for $i =1,\ldots,r$, where for every $j\in\{1,\ldots,m\}$ the map $pr_j:B\rightarrow B_j$ is the projection on the $j$-th factor, $\alpha_j$ is the indivisible class in $H^2(B_j;\mathbb{Z})$ such that $c_1(B_j) = q_j\alpha_j$ for $q_j\in\mathbb{Z}_+\setminus \{0\}$ and $[a_{ij}]$ is a $r\times m$ matrix with integer coefficients. Then $P$ with a metric defined by
\begin{displaymath}
g(X,Y) = \sum_{i,j=1}^r{b_{ij}\theta^i(X)\theta^j(Y)} + p^*h(X,Y),
\end{displaymath}
is an $\mathcal{A}$-manifold, where $\theta^i$ are as before and $[b_{ij}]$ is some positive, symmetric and non-singular matrix of dimension $r\times r$.
\end{thm}
\begin{proof}
We assume that $i=1,\ldots,r$ and $j=1,\ldots,m$ along the proof. Moreover, we normalize each metric $g_j$ so that the K\"ahler form $\eta_j$ of $g_j$ is in the cohomology class of $2\pi\alpha_j$, i.e. we rescale the metric so that each Einstein's constant is equal to $q_j$.

Let $P_i$ be the principal $S^1$-bundle over $B$ determined by $\beta_i$. From Theorem 1.4 \cite{wang} we know, that there exist a connection form $\theta_i$ on $P_i$ such that
\begin{displaymath}
d\theta_i = \sum_{k=1}^m{a_{ik}p^*\eta^*_k},
\end{displaymath}
where $\eta^*_j = pr^*_j\eta_j$. Since a K\"ahler form on a K\"ahler manifold is harmonic we see that the condition \eqref{orto} is satisfied.

Let $P$ be the principal $T^r$ bundle obtained as the Whitney sum of $P_1,\ldots,P_r$. We know that $\theta_1,\ldots,\theta_r$ are components of the principal connection on $P$ and $\Omega_i = d\theta_i$ is the $i$-th component of the curvature form on $P$. Recall that we have $\Omega_i = p^*\omega_i$ for some $2$-form on $B$. Comparing the above formula and \eqref{o1} for each $i$ we have
\begin{displaymath}
\sum_{k=1}^m{a_{ik}\eta^*_k(X,Y)} = 2\sum_{l=1}^r{b^{il}h(\tilde{T}_l(X),Y)},
\end{displaymath}
where $X,Y\in TB$. Because $h$ is the product metric, the complex structure tensor $J_j$ is an endomorphism of $TB_j$ and since $TB = TB_1\oplus\ldots\oplus TB_m$ the pull-back $J^*_j$ of $J_j$ by $pr_j$ preserves $TB_j$. Hence we can write $\eta^*_j(X,Y) = \frac{1}{x_j}h(J^*_jX,Y)$. We can now define $\tilde{T}_i$ with the following formula
\begin{displaymath}
\tilde{T}_i = \frac{1}{2}\sum_{k=1}^r{b_{ik}\left(\sum_{l=1}^m{\frac{1}{x_l}a_{kl}J^*_l}\right)},
\end{displaymath}
Since $\sum_{k=1}^r{b_{ik}a_{kl}}$ are just coefficients of the product of $[b_{ik}]$ and $[a_{ij}]$ which we denote by $c_{il}$ we can write the above formula as
\begin{displaymath}
\tilde{T}_i = \frac{1}{2}\sum_{l=1}^m{\frac{1}{x_l}c_{il}J^*_l}.
\end{displaymath}
We have to determine what conditions have to be imposed on $x_j$ for $(P,g)$ to be an $\mathcal{A}$-manifold. Let us denote by $\{X_j\}$ some orthonormal basis of $B$ such that $\{X_l^j\}_{l=1}^{n_j}$ is an orthogonal basis of $B_j$, where $n_j$ is the real dimension of $B_j$. For any $i,l\in\{1,\ldots,r\}$ we have
\begin{align*}
\sum_{k=1}^m{h(\tilde{T}_iX_k,\tilde{T}_lX_k)} &= \frac{1}{4}\sum_{k=1}^m{h\left(\sum_{p=1}^m{\frac{1}{x_p}c_{ip}J^*_pX_k},\sum_{s=1}^m{\frac{1}{x_s}c_{ls}J^*_sX_k}\right)} \\
&= \frac{1}{4}\sum_{k=1}^m{\sum_{p=1}^m{\frac{c_{ip}c_{lp}}{x_p^2}h\left(J^*_pX_k,J^*_pX_k\right)}},
\end{align*}
where the last equality follows from the above discussion of tensor $J_j$. We can descend further to metrics on $B_j$:
\begin{align*}
\sum_{k=1}^m{h(\tilde{T}_iX_k,\tilde{T}_lX_k)} &= \frac{1}{4}\sum_{j=1}^m\sum_{p=1}^m\sum_{k=1}^{n_j}\frac{c_{ip}c_{lp}}{x_p}g_j(J_pX_k,J_pX_k) \\
& = \frac{1}{4}\sum_{j=1}^m{\sum_{p=1}^m{\frac{c_{ip}c_{lp}}{x_p^2}n_j}}
\end{align*}
and we see that the condition \eqref{lambda} is satisfied.

Let us look at the equation \eqref{mu} defining the other eigenvalue.
\begin{align*}
\rho(X^*,X^*) &= \check{\rho}(X^*,X^*) - 2\sum_{i,k=1}^r{b^{ik}g(T_iX^*,T_kX^*)} \\
&= \rho_B(X,X) - 2\sum_{i,k=1}^r{b^{ik}h(\tilde{T}_iX,\tilde{T}_kX)}\\ 
&= \rho_B(X,X) - \frac{1}{2}\sum_{i,k=1}^r{b^{ik}\left(\sum_{p=1}^m\frac{1}{x_p^2}c_{ip}c_{kp}h(J^*_pX,J^*_pX)\right)}. 
\end{align*}
Assuming that $X$ is an element of the local orthonormal frame on $B_j$ we get $m$ equations
\begin{equation}\label{mulokal}
\mu = \frac{q_j}{x_j} - \frac{1}{2}\sum_{i,k=1}^r\frac{b^{ik}c_{ij}c_{kj}}{x_j^2},\quad j=1,\ldots,m.
\end{equation}
It is now easy to see that we can choose coefficients $x_j$ so that above equations are satisfied for any $a_{ik}$, where $k\in\{1,\ldots,m\}$, and $b_{ip}$, where $p\in\{1,\ldots,r\}$. One solution can be obtained if we take $x_s = \alpha_s x_1$ for some $\alpha_s > 0$, where $s\in\{2,\ldots,m\}$. In this case we have $m-1$ equations
\begin{equation}\label{warmi}
\frac{\alpha_s q_1 - q_s}{\alpha_s x_1} = \frac{1}{2}\sum_{i,k=1}^r\frac{\alpha_s^2b^{ik}c_{k1}c_{i1} - b^{ik}c_{ks}c_{is}}{\alpha_s^2 x_1^2},
\end{equation}
or
\begin{equation}
\frac{q_s}{\alpha_s} = q_1
\end{equation}
if the sum in the right-hand side of \eqref{mulokal} is zero. In the second case we just need to take $\alpha_s = \frac{q_s}{q_1}$ for $s=2,\ldots,m$. Equations \eqref{warmi} are equivalent to
\begin{displaymath}
\alpha_s x_1 = \frac{1}{2}\sum_{i,k=1}^r\frac{\alpha_s^2b^{ik}c_{k1}c_{i1} - b^{ik}c_{ks}c_{is}}{\alpha_s q_1 - q_s},
\end{displaymath}
where $s=2,\ldots,m$. Hence the positive solution exists iff the following equation has a positive solution
\begin{equation}\label{rk}
\left(2q_1x_1 -\sum_{i,k=1}^rb^{ik}c_{k1}c_{i1}\right)\alpha_s^2 - 2\alpha_sq_sx_1 + \sum_{i,k=1}^rb^{ik}c_{ks}c_{is} = 0.
\end{equation}
This equation has exactly two solutions precisely when the following inequality holds
\begin{displaymath}
q_s^2x_1^2 - \left(2q_1x_1 -\sum_{i,k=1}^rb^{ik}c_{k1}c_{i1}\right)\left(\sum_{i,k=1}^rb^{ik}c_{ks}c_{is}\right) >0.
\end{displaymath}
we see that if we choose $x_1$ big enough then this inequality is satisfied. Moreover, from Viete's formulae we have that the two solutions $\alpha_s'$ and $\alpha_s''$ satisfy
\begin{displaymath}
\alpha_s'+\alpha_s'' = \frac{2q_sx_1}{2q_1x_1 -\sum_{i,k=1}^rb^{ik}c_{k1}c_{i1}}.
\end{displaymath}
Again if we choose big $x_1$ we see that this sum is positive, hence at least one solution of \eqref{rk} is positive. Concluding, for big enough $x_1$ the system \eqref{warmi} has a positive solution.
 
We see that the assumptions of the Theorem \ref{a-tw} are satisfied, hence $(P,g)$ is an $\mathcal{A}$-manifold.
\end{proof}

\begin{rem}
Under further constraints we can observe that some of our $\mathcal{A}$-manifolds are proper. Consider the above construction on $T^r$-principal bundle with matrix $[a_{ij}]$ of rank $r$. Now observe that $\nabla\ric = 0$ implies $\nabla\eta_i = 0$, where $\eta_i$ is an eigenvector of the Ricci tensor corresponding to the eigenvalue $\lambda_i$ on the vertical distribution. Indeed let $X$ be in the horizontal distribution. Then
\begin{displaymath}
0 = \nabla_X\ric(\eta_i) = - (\ric -\lambda_i)\nabla_X\eta_i,
\end{displaymath}
which means that $\nabla_X\eta_i$ is in the vertical distribution, specifically in the eigendistribution corresponding to the eigenvalue $\lambda_i$. Since $\eta_i$ is a linear combination of $\xi^j$'s with constant coefficients, $\nabla_X\eta_i$ is horizontal. Hence it is zero.

From the fact that $\eta_i = \sum_{p=1}^re^i_p\xi^p$ and $e^i_p$ being all constants we see that $\nabla_X\eta_i = 0$ iff $\sum_{p=1}^re^i_p\nabla_X\xi^p =0$. We show that $\nabla_X\xi^p$'s are linearly independent for $p=1,\ldots,r$ and any horizontal vector field $X$ which is a contradiction with $\eta_i$ being non-zero.

Suppose that for some functions $b_1,\ldots,b_r$ and horizontal vector field $X$ we have
\begin{equation}\label{linind}
b_1\nabla_X\xi^1 +\ldots + b_r\nabla_X\xi^r = 0.
\end{equation}
Recall that in the setting from the above theorem 
\begin{displaymath}
\nabla_X\xi^p = \frac{1}{2}\sum_{l=1}^m\frac{c_{pl}}{x_l}\tilde{J}_l^*X,
\end{displaymath}
where $\tilde{J}_l^*$ is a tensor field obtained from lifting the complex structure tensor field $J_l$ of the K\"ahler manifold $M_l$ to $P$. Now equation \eqref{linind} becomes
\begin{displaymath}
\sum_{l=1}^m\sum_{p=1}^r\left(b_p\frac{c_{pl}}{x_l}\right)\tilde{J}^*_lX = 0.
\end{displaymath}
If we take $X$ to be some element $E_l$ of the orthonormal basis of $M_l$ then for each $l=1,\ldots,m$ we get
\begin{displaymath}
\sum_{p=1}^rb_p\frac{c_{pl}}{x_l}\tilde{J}^*_lE_l = 0
\end{displaymath}
and this is equivalent to $\sum_{p=1}^rb_pc_{pl} = 0$ for each $l$. We can write this in the matrix form as $C^Tb = 0$ where $C^T$ is the transposed matrix of $C=[c_{ij}]$ and $b$ is the column vector $[b_1,\ldots,b_r]^T$. Recall that $C = BA$, where $B$ is the matrix of the metric on torus $T^r$ and $A=[a_{ij}]$, as defined in the above theorem, is of rank $r$, hence $C$ is of rank $r$. Now we can see that $C^Tb = 0$ iff $b=0$ and it follows that $\nabla\xi^1,\ldots,\nabla\xi^r$ are linearly independent, which is a contradiction with $\nabla\eta_i = 0$ for any $i=1,\ldots,r$.
\end{rem}


\begin{thebibliography}{99}
\bibitem[Be]{besse} A. Besse, \textit{Einstein Manifolds}, Springer-Verlag, Berlin, Heidelberg (1987).
\bibitem[Gra]{gray} A. Gray, \textit{Einstein-like manifolds which are not Einstein}, Geom. Dedicata 7 (1978), 259--280.
\bibitem[Jel1]{j1} W. Jelonek, \textit{On $\mathcal{A}$-tensors in Riemannian geometry}, preprint PAN, 551, 1995.
\bibitem[Ko1]{koba1} S. Kobayashi, \textit{On compact K\"ahler manifolds with positive definite Ricci tensor}, Ann. of Math. 74 (1961), 570--574.
\bibitem[Ko2]{koba2} S. Kobayashi, \textit{Principal fibre bundles with the $1$-dimensional toroidal group},  Tohoku Math. J. 8 (1956), 29--45.
\bibitem[ON]{oneill} B. O'Neill, \textit{The fundamental equations of a submersion}, Michigan Math. J. 13 (1966), 459--469.
\bibitem[PT]{pet} H. Pedersen, P. Todd, \textit{The Ledger curvature conditions and D'Atri geometry}, Differential Geom. Appl. 11 (1999), 155--162.
\bibitem[W-Z]{wang} M. Y. Wang, W. Ziller, \textit{Einstein metrics on torus bundles}, J. Differential Geom. 31 (1990), 215--248.
\end{thebibliography}
\end{document}